\documentclass[10pt]{amsart}
\usepackage{amssymb,amsmath}
\usepackage{calrsfs}
\usepackage{verbatim}
\usepackage[mathscr]{eucal}
\usepackage{pgf}
\def\mL{L\kern-0.08cm\char39}
\def\ml{l\kern-0.0035cm\char39\kern-0.03cm}
\newtheorem{thm}{Theorem}[section]
\newtheorem{lemma}[thm]{Lemma}
\newtheorem{claim}{\sc Claim}[thm]
\newtheorem*{claimnn}{\sc Claim}
\newtheorem{ex}[thm]{Example}
\newtheorem{prop}[thm]{Proposition}
\newtheorem{cor}[thm]{Corollary}
\theoremstyle{definition}

\theoremstyle{remark}

\begin{document}

\title{Completeness and related properties of the graph topology on function spaces}
\author{\mL ubica Hol\'a and L\'aszl\'o Zsilinszky}
\address{Academy of Sciences, Institute of Mathematics, \v
Stef\'anikova 49, 81473 Bratislava, Slovakia}
\email{hola@mat.savba.sk}
\address{Department of Mathematics and Computer Science, The University of
North Carolina at Pembroke, Pembroke, NC 28372, USA}
\email{laszlo@uncp.edu}

\subjclass[2010]{Primary 54C35; Secondary 54E52, 54B20, 54C05}

\keywords{graph topology, fine topology, hereditarily Baire space, Banach-Mazur game, strong Choquet game, (strongly) $\alpha$-favorable space, pseudocomplete space, countable compact, pseudocompact}

\thanks{The first author would like to thank the support of the grant APVV-0269-11 and  Vega 2/0018/13.}

\begin{abstract} The graph topology $\tau_{\Gamma}$ is the topology on the space $C(X)$ of all continuous functions defined on a Tychonoff space $X$ inherited from the Vietoris topology on $X\times \mathbb R$ after identifying continuous functions with their graphs. It is shown that all  completeness  properties between complete metrizability and hereditary Baireness coincide for the graph topology if and only if $X$ is countably compact;  however, the graph topology is  $\alpha$-favorable in the strong Choquet game, regardless of $X$.  Analogous results are obtained for the fine topology on $C(X)$. Pseudocompleteness,  along with  properties related to 1st and 2nd countability of $(C(X),\tau_{\Gamma})$ are also investigated.

\end{abstract}

\maketitle

\section{Preliminaries}
There have been a plethora of topologies studied on the space $C(X,Y)$ of continuous functions $f:X\to Y$; most of these topologies can be described as having base elements  of the form
\[F_U=\{f\in C(X,Y): \text{graph}(f)\subseteq U\},\]
where $U$ ranges over a specific collection of open subsets of $X\times Y$.
The best known topologies of this kind are the topology of pointwise convergence $\tau_p$, the compact-open topology $\tau_k$, the topology of uniform convergence $\tau_u$, and the fine topology  $\tau_w$ \cite{No}. We will be interested in the finest of these topologies, the so-called {\it graph topology} $\tau_{\Gamma}$ of {\it Naimpally} \cite{Na}, which is generated by the sets $F_U$ when $U$ ranges over all the nonempty open subsets  of $X\times Y$. The symbol $C_{\Gamma}(X,Y)$ will stand for the space $(C(X,Y),\tau_{\Gamma})$, and we will write $C_{\Gamma}(X)$ instead of  $C_{\Gamma}(X,\mathbb R)$. Denote by $cl_{\Gamma}(A)$ the $\tau_{\Gamma}$-closure of $A\subset C(X,Y)$.
It is the purpose of this paper to investigate completeness and related topological properties
 of the graph topology. It will be shown (Theorem  \ref{hbg}) that $C_{\Gamma}(X)$ is  hereditarily Baire iff $C_{\Gamma}(X)$ is completely metrizable iff $X$ is countably complete; so even the weakest closed hereditary  completeness property - hereditary Baireness - of $C_{\Gamma}(X)$ imposes a strong restriction on $X$, however, there is another strong (non-closed hereditary) completeness property, that of strong $\alpha$-favorability, which $C_{\Gamma}(X)$ always possesses regardless of $X$ (Theorem \ref{choquetg}). Analogous results are established for the  fine topology on $C(X)$ (Theorems \ref{hbw}, \ref{choquetw}).
Various (non-completeness) properties, from Arhange\ml skii's $p$-spaces property to sequentiality, countable tightness, or the $k$-space property for $C_{\Gamma}(X)$ are shown to be  equivalent to  (complete) metrizability of $C_{\Gamma}(X)$ (Theorem \ref{metr}).
As a byproduct of these results, strongly $\alpha$-favorable spaces can be constructed lacking all the discussed topological properties.
Pseudocompleteness, along with properties related to the 2nd countability of the graph topology are also investigated. 

There are various ways of looking at the graph topology, for example,  if $X$ is $T_1$, and $Y$ is $T_2$, then $C_{\Gamma}(X,Y)$  is the relative Vietoris topology \cite{Mi} inherited from the hyperspace of nonempty closed subsets of $X\times Y$ after identifying  elements of $C(X,Y)$ with their graphs. Indeed, since the $F_U$'s form the upper Vietoris topology on $C(X,Y)$, we just need to show that a typical lower Vietoris open set is open in $C_{\Gamma}(X,Y)$: let $V,W$ be open subsets of $X,Y$ respectively,  then
\[\{f\in C(X,Y): \text{graph}(f)\cap (V\times W)\neq\emptyset\}=\bigcup_{x\in V} F_{(X\times W)\cup ((X\setminus \{x\})\times Y)}\in\tau_{\Gamma}.\]
Note that  other  hyperspace topologies also coincide with the graph topology on $C(X,Y)$, as was demonstrated in \cite{BN}.

Another useful way of looking at $\tau_{\Gamma}$ was first observed by {\it van Douwen} \cite[Lemma 8.3.]{vD} for $C_{\Gamma}(X)$, since his proof works for any metrizable $Y$, we will state it in this more general form.
First introduce some notation: denote by $C^+(X)$ (resp. $LSC^+(X)$), the strictly positive real-valued continuous (resp. lower semicontinuous) functions defined on the topological space $X$. Given a function $\varepsilon:X\to (0,\infty)$, a metric space $(Y,d)$, and $f\in C(X,Y)$, define
\[B(f,\varepsilon) = \{g \in C(X, Y ) :  d(f(x), g(x)) < \varepsilon(x) \ \text{for all} \ x\in X\}.\]

\begin{prop}\label{graphbase}
Let $X$ be a topological space, and $(Y,d)$ a metric space. The collection
\[\mathcal B_{\Gamma}=\{B(f,\varepsilon): f\in C(X,Y), \varepsilon\in LSC^+(X)\}\] is a base for $C_{\Gamma}(X,Y)$.
\end{prop}

The previous result shows how close the graph topology is  to the  {\it fine topology} $\tau_w$ on $C(X,Y)$ with $(Y,d)$ metric, in short denoted as $C_w(X,Y)$; indeed,  $C_w(X,Y)$ also has base elements  of the form $B(f,\varepsilon)$, but with $\varepsilon\in C^+(X)$. The fine topology (also called $m$-topology \cite{vD}) has been thoroughly investigated in the past, see \cite{MN}, \cite{MC1}, \cite{DHHM}, \cite{HM}. In particular, it is known that $C_w(X,Y)$ is sensitive to the metric $d$ of the range space $Y$ \cite{MN}, which immediately shows a difference with the graph topology, as $\tau_{\Gamma}$ is clearly independent of the compatible metrics of $Y$ (for  completeness, when $Y=\mathbb R$, we will assume that $\mathbb R$ carries the Euclidean metric). Moreover, the following is not hard to see:

\begin{prop}\cite{vD}\label{w=g} Let $X$ be a topological space, and $Y$ a metric space.
The following are equivalent:
\begin{enumerate}
\item $C_w(X,Y)=C_{\Gamma}(X,Y)$,
\item $X$ is a $cb$-space, i.e. for each $\varepsilon\in LSC^+(X)$ there is some $\varphi\in C^+(X)$ with $\varepsilon(x)\ge \varphi(x)$ for all $x\in X$.
\end{enumerate}
\end{prop}
Note that a normal space is a $cb$-space iff it is countably paracompact \cite{Ma}, and there are non-normal, locally compact, countably paracompact, non-$cb$-spaces \cite[p. 240]{MJ}.

It was shown in \cite{Na}, that if $X$ is $T_1$, and $Y$ contains at least two points, then $C_{\Gamma}(X,Y)$ is $T_1$ (resp. $T_2$) iff $Y$ is  $T_1$ (resp. $T_2$). There are no such nice characterizations available for regularity of $C_{\Gamma}(X,Y)$, it is easy to see that regularity of $Y$ is necessary for regularity of $C_{\Gamma}(X,Y)$,  but it is not sufficient:

\begin{ex}
Let $X=\omega_1$, and $Y=\omega_1+1$, both with the order topology. Then $C_{\Gamma}(X,Y)$ is not regular.
\end{ex}

\begin{proof}
Let $f:X\to Y$ be the identity function, and consider its $\tau_{\Gamma}$-neighborhood $F_{\omega_1\times\omega_1}$. Let $U\subseteq\omega_1\times\omega_1$ be any open subset with   $f\in F_{U}$. Since the graph of $f$ is the diagonal of $\omega_1\times\omega_1$, we can find an $\alpha<\omega_1$ so that $(x,y)\in U$ whenever $x,y>\alpha$. Define $g\in C(X,Y)$ via
\[g(x)=\begin{cases} x, &\text{if} \ x\neq \alpha+1
\\ \omega_1, &\text{if} \ x= \alpha+1.\end{cases}
\]
It is not hard to see that $g\in cl_{\Gamma}(F_U) \setminus F_{\omega_1\times\omega_1}$.
\end{proof}

As for an easy sufficient condition for regularity of $C_{\Gamma}(X,Y)$ we can  assume that $X\times Y$ is $T_4$, since then the entire hyperspace of the nonempty closed subsets of $X\times Y$ with the Vietoris topology is regular \cite{Mi}. If $X\times Y$ is non-normal, $C_{\Gamma}(X,Y)$ may be non-regular (see the previous example), but it also can be regular (since, if $X$ is compact and $Y$ is regular, then  $C_{\Gamma}(X,Y)$ is regular \cite{Le}). Fortunately, most of our results concern $C_{\Gamma}(X)$, which is a topological group, and hence a Tychonoff space, if $X$ is $T_1$.

\section{Hereditary Baireness  of the graph and fine topologies}

One can argue that, aside from compactness, the strongest closed hereditary property is complete metrizability, and the weakest such property is hereditary Baireness. Recall that $Z$ is {\it hereditary Baire} iff nonempty closed subspaces of $Z$ are of the 2nd category in themselves iff nonempty closed subspaces of $Z$ are Baire spaces; moreover, $Z$ is a {\it Baire space} \cite{HMC} iff nonempty open subsets are of the 2nd category in themselves iff a countable dense open collection in $Z$ has a dense intersection.

Extending results of \cite{HH}, and \cite{MC1}, \cite{DHHM}, we will show that for the graph and fine topologies these properties coincide, and so does any other closed hereditary completeness property in-between them. We just include two such well-studied properties, namely   {\it \v Cech completeness}  (being a $G_{\delta}$-subspace in a compactification \cite{En}), and {\it sieve completeness} (being a continuous open image of a \v Cech complete space \cite{WW}).

\begin{thm}\label{hbg}
 Let $X$ be a Tychonoff space. The following are equivalent:
 \begin{enumerate}
 \item $C_{\Gamma}(X)$ is completely metrizable,
 \item $C_{\Gamma}(X)$ is \v Cech complete,
 \item $C_{\Gamma}(X)$ is sieve complete,
 \item  $C_{\Gamma}(X)$ is hereditarily Baire,
 \item $X$ is countably compact.
 \end{enumerate}
 \end{thm}

\begin{proof} The implications (1)$\Rightarrow$(2)$\Rightarrow$(3)$\Rightarrow$(4) are well-known.

(4)$\Rightarrow$(5)
If $X$ is not countably compact, there is a  countable closed discrete set $\{x_n: n<\omega\}$. For each $n<\omega$ define
\[H_n=\{f\in C(X): \ \forall k\ge n, \ f(x_k)=0 \},\]
and put $H=\bigcup_{n<\omega} H_n$. Then

$\bullet$ {\it  $H$ is closed in $C_{\Gamma}(X)$}:
if  $f\in C(X)\setminus H$, then $f(x_n)\neq 0$ for  infinitely many $n$ (w.l.o.g., all $n$). Define
$U=X\times \mathbb R \setminus \{x_n:n<\omega\}\times \{0\}$; then $f\in F_U\subseteq C(X)\setminus H$.

$\bullet$ {\it  each $H_n$ is nowhere dense in $H$}:  it is easy to see that $H_n$ is closed in $C_{\Gamma}(X)$;
moreover, assume there is  an $X\times\mathbb R$-open $V$ such that $\emptyset\neq F_V\cap H\subseteq H_n$, and pick $f\in F_V\cap H_n$. Then $f(x_n)=0$, so we can find $\varepsilon>0$ and an $X$-open neighborhood $U$ of $x_n$  missing $\{x_k: k\neq n\}$ such that
\[U\times (-\varepsilon,\varepsilon) \subseteq V \ \text{and} \ \overline{U}\subseteq f^{-1}((-\frac{\varepsilon}{2},\frac{\varepsilon}{2})).\]
Define a continuous function $g_0:X\to [0,\frac{\varepsilon}{2}]$ so that $g_0(x_n)=\frac{\varepsilon}{2}$, and $g_0(x)=0$ whenever $x\notin U$. Then $g=f+g_0\in F_V\cap H\setminus H_n$, a contradiction.

(5)$\Rightarrow$(1)  If $X$ is countably compact, then each $\varepsilon\in LSC^+(X)$ has a (positive) minimum, so  $C_{\Gamma}(X)$ coincides with the  uniform topology on $C(X)$, which is completely metrizable by the sup-metric.
\end{proof}

Results about hereditary Baireness of function spaces are very rare and just partial \cite{GS}, \cite{Bo}. We will  show that it is possible to use the idea of Theorem \ref{hbg} to completely characterize hereditary Baireness of the fine topology as well:

\begin{thm}\label{hbw}
 Let $X$ be a Tychonoff space. The following are equivalent:
 \begin{enumerate}
 \item $C_w(X)$ is completely metrizable,
 \item $C_{w}(X)$ is \v Cech complete,
 \item $C_{w}(X)$ is sieve complete,
 \item  $C_w(X)$ is hereditarily Baire,
 \item $X$ is pseudocompact.
 \end{enumerate}
 \end{thm}

\begin{proof} The implications (1)$\Rightarrow$(2)$\Rightarrow$(3)$\Rightarrow$(4) are well-known.

 (4)$\Rightarrow$(5)
If $X$ is not pseudocompact, there is a  countable collection  $\{U_n:n<\omega\}$ of open sets such that $\{\overline{U_n}:n<\omega\}$ is discrete. For $n<\omega$ fix some $x_n\in U_n$, and define
\[H_n=\{f\in C(X): \ \forall k\ge n \ f(U_k)=\{0\} \}.\]

\begin{claim}
$H_n$ is nowhere dense in $H=\bigcup_{n<\omega} H_n$:
\end{claim}

It is easy to see that $H_n$ is closed in $C_w(X)$. Let $f\in H_n$, and consider $B(f,\varepsilon)$ for some $\varepsilon\in C^+(X)$. Then $V=U_n\cap \varepsilon^{-1}((\frac23 \varepsilon(x_n),\infty))$ is an open neighborhood of $x_n$. Let $g_0:X\to [0,\frac{\varepsilon(x_n)}{2}]$ be a continuous function, such that $g_0(x_n)=\frac{\varepsilon(x_n)}{2}$, and $g_0(x)=0$ whenever $x\notin V$; then $g=f+g_0\in H_{n+1}$, since $g(x_n)=g_0(x_n)>0$, and $g=f$ outside of $V$, so $g(U_k)=\{0\}$ whenever $k\ge n+1$. We have

$\bullet$ $B(g,\frac{\varepsilon}{4})\subseteq B(f,\varepsilon)$

[Let $h\in B(g,\frac{\varepsilon}{4})$. If $x\notin V$, then $|h(x)-f(x)|=|h(x)-g(x)|<\frac{\varepsilon(x)}{4}<\varepsilon(x)$,
so $h\in B(f,\varepsilon)$. If $x\in V$,  then
$|h(x)-f(x)|\le |h(x)-g(x)|+|g(x)-f(x)|<\frac{\varepsilon(x)}{4}+g_0(x)\le \frac{\varepsilon(x)}{4}+\frac{\varepsilon(x_n)}{2}<\varepsilon(x)$,
so $h\in B(f,\varepsilon)$ again.]

$\bullet$ $ B(g,\frac{\varepsilon}{4})$ is disjoint to $H_n$

[If $h\in  B(g,\frac{\varepsilon}{4})$, then $h(x_n)>g(x_n)-\frac{\varepsilon(x_n)}{4}=\frac{\varepsilon(x_n)}{2}-\frac{\varepsilon(x_n)}{4}=\frac{\varepsilon(x_n)}{4}>0$, so $h\notin H_n$.]

\begin{claim}
$H$ is closed in $C_w(X)$:
\end{claim}
let $f\in C(X)\setminus H$. Then $f(U_n)\neq\{0\}$ for  infinitely many $n$ (w.o.l.g., all $n$). Let $u_n\in U_n$ be such that $f(u_n)\neq 0$ for all $n$. Given $n$, define the continuous function $\varepsilon_n:\overline U_n\to [\frac12,\frac{|f(u_n)|}{2}]$ if $1\le |f(u_n)|$, or $\varepsilon_n:\overline U_n\to [\frac{|f(u_n)|}{2},\frac12]$ if $1> |f(u_n)|$ so that $\varepsilon_n(u_n)= \frac{|f(u_n)|}{2}$, and $\varepsilon_n(\overline{U_n}\setminus U_n)=\{\frac12\}$. Finally, define $\varepsilon\in C^+(X)$ as follows:
\[\varepsilon=
\begin{cases}
\varepsilon_n, \ &\text{on} \ \overline{U_n}, \ \text{whenever} \ n<\omega,\\
\frac12, \ &\text{on} \ X\setminus \bigcup_n\overline{U_n}.
\end{cases}
\]
If $h\in B(f,\varepsilon)$, then for each $n\in\omega$
\[|h(u_n)|\ge   |f(u_n)| -|f(u_n)-h(u_n)|>|f(u_n)| - \varepsilon_n(u_n)=|f(u_n)| - \frac{|f(u_n)|}{2}=\frac{|f(u_n)|}{2}>0,\]
so $h\notin H$; thus, $f\in B(f,\varepsilon)\subseteq C(X)\setminus H$.

(5)$\Rightarrow$(1)  If $X$ is pseudocompact, then each $\varepsilon\in C^+(X)$ is bounded away from zero, so   $C_w(X)$ coincides with the  uniform topology on $C(X)$.
\end{proof}

\section{Topological games and the graph and fine topologies}

In the {\it strong Choquet game }  $Ch(Z)$ (cf. \cite{Ch}, \cite{Ke}) players
$\alpha$ and $\beta$ take turns in choosing objects in the topological
space $Z$ with an open base $\mathcal  B$: $\beta$ starts by picking $(z_0,V_0)$
from
$\mathcal E=\{ (z,V)\in Z\times\mathcal  B: \ z\in V\}$
and $\alpha$ responds by
$U_0\in \mathcal  B$ with $z_0\in U_0\subseteq V_0$. The next choice of
$\beta$ is  $(z_1,V_1)\in \mathcal  E$ with $V_1\subset U_0$
and again $\alpha$ picks $U_1$ with $z_1\in U_1\subseteq V_1$ etc. Player
$\alpha$ wins the run $(z_0,V_0),U_0,\dots,(z_n,V_n),U_n,\dots$ provided
$\bigcap_{n} U_n=\bigcap_{n} V_n\neq \emptyset$,
otherwise $\beta$ wins. The space $Z$ is called { \it strongly $\alpha$-favorable}, provided $Ch(Z)$ is $\alpha$-favorable, i.e. when
$\alpha$ has a winning tactic in $Ch(Z)$, which
is a function $t:\mathcal  E\to\mathcal  B$ such that
$\alpha$ wins every run of the game with $U_n=t(z_n,V_n)$ for all $n$.
The {\it Banach-Mazur game} $BM(Z)$ (see \cite{HM}, or the Choquet game in \cite{Ke})
is played as the strong Choquet game, except
 $\beta$'s  choice is only a nonempty open set contained in the previous
choice of $\alpha$. A space $Z$ is  { \it  $\alpha$-favorable}, provided
$\alpha$ has a  winning tactic in $BM(Z)$.
$\beta$-favorability of $Ch(Z)$, and $BM(Z)$ can be defined analogously \cite{Zs1}, \cite{HM}.

The strong Choquet game  is intimately related to the completeness properties considered in Theorems \ref{hbg}, \ref{hbw}: indeed, a Moore space is sieve complete iff it is strongly $\alpha$-favorable \cite[Corollary 3.2]{CP}; a metrizable space is completely metrizable iff  it is strongly $\alpha$-favorable \cite{Ch} (Choquet's Theorem). On the other hand, a Moore space is hereditarily Baire iff the strong Choquet game is not $\beta$-favorable (see \cite[Corollary 3.3]{Zs1}, or \cite{De} for metrizable spaces). Note however, that neither  $\alpha$-favorability, nor  $\beta$-unfavorability of the strong Choquet game is closed hereditary (the Michael line is one example, see \cite{Zs1} for more).
It is clear that strong $\alpha$-favorability implies $\alpha$-favorability, which in turn implies Baireness, since  $Z$ is a Baire space iff $BM(Z)$ is not $\beta$-favorable \cite{Ke}.

\begin{thm}\label{choquetg}
 Let $X$ be  a topological space, and $Y$ be metrizable. The following are equivalent:
 \begin{enumerate}
 \item $C_{\Gamma}(X,Y)$ is strongly $\alpha$-favorable,
 \item $Y$ is completely metrizable.
 \end{enumerate}
 \end{thm}

\begin{proof} (1)$\Rightarrow$(2) Note that $Y$ embeds as a $G_{\delta}$ in  $C_{\Gamma}(X,Y)$: indeed, $Y$ embeds as a closed subspace of $(C(X,Y),\tau_p)$, thus also as a closed subspace of $(C(X,Y),\tau_u)$ and $\tau_u$ is metrizable. It follows that $Y$ is a strongly $\alpha$-favorable metrizable space, which is equivalent to complete metrizability of $Y$ by Choquet's Theorem \cite{Ch}.

(2)$\Rightarrow$(1) Let $d$ be a compatible complete metric for $Y$. Define a tactic $t$ for $\alpha$ in the strong Choquet game on $C_{\Gamma}(X,Y)$ as  follows: given $U=(f,B(f,\varepsilon))$, put $t(U)=B(f,\frac{\min(1,\varepsilon)}{2})$.

\begin{claimnn}
$t$ is a winning tactic for $\alpha$
\end{claimnn}
\noindent
Let $(f_0,B_0(f_0,\varepsilon_0)),t(f_0,B_0(f_0,\varepsilon_0)),\dots (f_n,B_n(f_n,\varepsilon_n)),t(f_n,B_n(f_n,\varepsilon_n)), \dots$ be a run of the strong Choquet game compatible with $t$. Since $d$ is complete, the sequence $(f_n)$ uniformly converges to some $f\in C(X,Y)$. Also, $f_m\in t(B_n(f_n,\varepsilon_n))\subseteq B_n(f_{n},\frac{\varepsilon_{n}}{2})$ for each $m\ge n$, so $d(f_n(x),f_m(x))<\frac{\varepsilon_{n}(x)}{2}$ for all $x\in X$. Fix some $x\in X$, and choose $m_0\ge n$ with $d(f_{m_0}(x),f(x))<\frac{\varepsilon_n(x)}{2}$. Then
\[d(f(x),f_n(x))\le d(f(x),f_{m_0}(x))+d(f_{m_0}(x),f_n(x))<\frac{\varepsilon_n(x)}{2}+\frac{\varepsilon_n(x)}{2}=\varepsilon_n(x),\]
so $f\in B_n(f_n,\varepsilon_n)$ for each $n$; thus, $\alpha$ wins.
\end{proof}

A similar argument also gives:

\begin{thm}\label{choquetw}
 Let $X$ be  a topological space, and $Y$ be a complete metric space. Then $C_{w}(X,Y)$ is strongly $\alpha$-favorable.
 \end{thm}

\begin{cor}\label{alpha} Let $X$ be a topological space.
\begin{enumerate}
\item If $Y$ is completely metrizable, then $C_{\Gamma}(X,Y)$ is $\alpha$-favorable and, hence,   a Baire space.
\item If $Y$ is a complete metric space, then $C_w(X,Y)$ is $\alpha$-favorable and, hence,  a Baire space.
\end{enumerate}
\end{cor}

\section{Pseudocompleteness of the graph topology}

If we compare Theorems \ref{hbg}, \ref{hbw}, we can  fall under the impression that (closed hereditary) properties of the graph topology are to (closed hereditary)  properties of the fine topology as countable compactness is to pseudocompactness, which is indeed a good guiding idea when investigating these topologies. Moreover, Theorems \ref{choquetg}, \ref{choquetw} and Corollary \ref{alpha} suggest an even closer relationship between these topologies for some completeness properties. It was surprising therefore to find  that a property which is equivalent to $\alpha$-favorability in, say, Moore spaces \cite{Wh}, namely that of pseudocompleteness, is relatively hard to come by for $C_{\Gamma}(X)$, although it is not that complicated for $C_w(X)$ \cite[Theorem 3.2]{MC1}. Recall that a space $Z$ is  {\it pseudocomplete} \cite{Ox} iff  $Z$ is quasi-regular (i.e. each nonempty open set contains the closure of a nonempty open subset), and $Z$ has a sequence $\{\mathcal B_n\}_n$ of $\pi$-bases such that
if $B_n\in \mathcal B_n$ and $\overline{B_{n+1}}\subseteq B_n$ for all $n$, then $\bigcap_n B_n\neq\emptyset$.

\begin{lemma}\label{key}
Let $X$ be a Tychonoff space. Given $B(f,\varepsilon), B(g,\phi)\in\mathcal B_{\Gamma}$, consider the properties:
\begin{enumerate}
\item[{(i)}] $\text{cl}_{\Gamma}(B(f,\varepsilon))\subseteq B(g,\phi)$,
\item[{(ii)}] $[f(x)-\varepsilon(x),f(x)+\varepsilon(x)]\subseteq (g(x)-\phi(x),g(x)+\phi(x))$ for all $x\in X$.
\end{enumerate}
Then (i)$\Rightarrow$(ii) in each of the following cases:
\begin{enumerate}
\item points of $X$ are $G_{\delta}$,
\item $X$ is  locally countably compact.
\end{enumerate}
\end{lemma}

\begin{proof} Let $x_0\in X$, and fix $t\in [f(x_0)-\varepsilon(x_0),f(x_0)+\varepsilon(x_0)]$.

(a) If $t=f(x_0)$, then $f\in\text{cl}_{\Gamma}(B(f,\varepsilon))\subseteq B(g,\phi)$ implies $|g(x_0)-t|=|g(x_0)-f(x_0)|<\phi(x_0)$, so $t \in  (g(x_0)-\phi(x_0),g(x_0)+\phi(x_0))$.

(b) If $t>f(x_0)$, then $0<\delta=:t-f(x_0)\le \varepsilon(x_0)$. Put $U_0=X$, and choose a strictly decreasing sequence $(U_n)_{n\ge 1}$  of open neighborhoods of $x_0$ such that for each $n\ge 1$
\[U_n\subseteq \varepsilon^{-1}((\frac{(2^n-1)\delta}{2^n},\infty));\]
 moreover, find a continuous $g_{n}:X\to [0,\frac{\delta}{2^{n}}]$ such that
\[g_{n}(x_0)=\frac{\delta}{2^{n}} \ \text{and} \ g_{n}(X\setminus U_{n})=\{0\}.\] Inductively define the continuous functions  $h_n=h_{n-1}+g_n$, where $h_0=f$.
It follows that
\begin{itemize}
\item $h_n=h_{n-1}$ on $X\setminus U_n$,
\item $0\le h_n(x)-f(x)\le \frac{(2^n-1)\delta}{2^n}$ on $U_n$

\noindent
[proof by induction:
$h_n(x)-f(x)= (h_{n}(x)-h_{n-1}(x)) + (h_{n-1} (x)-f(x))\le $

$\le g_{n}(x) + \frac{(2^{n-1}-1)\delta}{2^{n-1}}\le \frac{\delta}{2^{n}} + \frac{(2^{n-1}-1)\delta}{2^{n-1}}=\frac{(2^n-1)\delta}{2^{n}}$]
\item $h_n(x_0)=f(x_0)+\frac{(2^n-1)\delta}{2^n}$

\noindent
[proof by induction:
$h_{n}(x_0)=h_{n-1}(x_0)+g_{n}(x_0)=f(x_0)+\frac{(2^{n-1}-1)\delta}{2^{n-1}}+\frac{\delta}{2^{n}}=f(x_0)+\frac{(2^n-1)\delta}{2^{n}}$]
\end{itemize}
\noindent
Also, if  $x\in X$ and $m>n\ge 1$, then
\[|h_{m}(x)-h_{n}(x)|=g_{m}(x)+\dots+ g_{n+1}(x)\le \frac{\delta}{2^{m}}+\dots + \frac{\delta}{2^{n+1}}\le  \frac{\delta}{2^{n}},\]
so $(h_n)$ (uniformly) converges to some $h\in C(X)$.

Denote $D=\bigcap_{n\ge 1} U_n$, and take $x\in D$. Then $\varepsilon(x)\ge \delta$ , so
\[h(x)=\lim_n h_n(x)\le \lim_n(f(x)+\frac{(2^n-1)\delta}{2^n})=f(x)+\delta\le f(x)+\varepsilon(x);\]
also, $h(x_0)=\lim_n h_n(x_0)=f(x_0)+\delta=t$. Moreover, if $x\notin D$, then
$x\in U_n\setminus U_{n+1}$ for some $n$. Then $h_m(x)=h_n(x)$ for all $m\ge n$, thus, $h(x)=h_n(x)$, and consequently,
\[
0\le h(x)-f(x)=h_n(x)-f(x)\le\frac{(2^n-1)\delta}{2^n}<\varepsilon(x), \ \text{if}  \ x\notin D. \tag{$\ast$}
\]
It also follows that
\[0\le h(x)-f(x)\le \min\{\delta,\varepsilon(x)\} \  \text{for each} \ x\in X.\tag{$\ast\ast$}\]

\begin{claimnn} $h\in\text{cl}_{\Gamma}(B(f,\varepsilon))$
\end{claimnn}

Take some  $B(h,\xi)\in\mathcal B_{\Gamma}$, and consider two cases:
\vskip 5pt
\noindent
(1) if the points of $X$ are $G_{\delta}$,  let $(G_n)$ be a sequence of open sets with $\{x_0\}=\bigcap_n G_n$, and for each $n\ge 1$ choose
\[U_n=G_n\cap \varepsilon^{-1}((\frac{(2^n-1)\delta}{2^n},\infty)).\]
Also find $m> 1$ with $\frac{\xi(x_0)}{m}\le \frac{\delta}{2}$, and define
\[V=\xi^{-1}((\frac{\xi(x_0)}{m},\infty))\cap \varepsilon^{-1}((\frac{\delta}{2},\infty)),\]
which is an open neighborhood of $x_0$.
Find a continuous function $k_0:X\to [0,\frac{\xi(x_0)}{m}]$ so that \[k_0(x_0)=\frac{\xi(x_0)}{m} \ \text{and} \ k_0(X\setminus V)=\{0\}.\]
We will be done, if we show that $k=h-k_0\in B(f,\varepsilon)\cap B(h,\xi)$:

$\bullet$  $k\in B(h,\xi)$:  if $x\in X\setminus V$, then $k(x)=h(x)$, so $|h(x)-k(x)|=0<\xi(x)$; if $x\in V$, then $|h(x)-k(x)|=k_0(x)\le \frac{\xi(x_0)}{m} <\xi(x)$.

$\bullet$ $k\in B(f,\varepsilon)$: first note that
\[k(x_0)-f(x_0)=(t-\frac{\xi(x_0)}{m})-(t-\delta)=\delta-\frac{\xi(x_0)}{m}\le \varepsilon(x_0)-\frac{\xi(x_0)}{m}<\varepsilon(x_0);\]
moreover, if $x\neq x_0$, using ($\ast$) we have
\[k(x)-f(x)=(h(x)-k_0(x))- f(x)=(h(x)-f(x))-k_0(x)<\varepsilon(x).\]
On the other side,
\[k(x)-f(x)=(h(x)-f(x))-k_0(x)
\begin{cases}
\ge -k_0(x)\ge -\frac{\xi(x_0)}{m}\ge-\frac{\delta}{2}>-\varepsilon(x), \ &\text{if} \ x\in V\\
\ge 0>-\varepsilon(x), \ &\text{if} \ x\notin V
\end{cases}
\]

\vskip 6pt
\noindent
(2) If $X$ is locally countably compact, let $(U_n)_{n\ge 1}$ be a sequence of open neighborhoods of $x_0$ with a countably compact closure such that whenever $n\ge 1$,
\[U_{n}\subseteq U_{n-1}\cap\varepsilon^{-1}((\frac{(2^{n}-1)\delta}{2^{n}},\infty)).
\]
Since $\overline{U_1}$ is countably compact, $\xi$ has a minimum on $\overline{U_1}$, so we can choose
\[0<r< \min\{\delta, \min\{\xi(x): x\in \overline{U_1}\}\}.\]
Define $k=h-\frac{r}{\delta}(h-f)$. The Claim will be proved, if we show that $k\in B(f,\varepsilon)\cap B(h,\xi)$:

$\bullet$ $k\in B(h,\xi)$, since, using ($\ast\ast$), we have
\[|h(x)-k(x)|=\frac{r}{\delta}|h(x)-f(x)|
\begin{cases}
\le  \frac{r}{\delta}\delta=r<\xi(x), \ &\text{if} \ x\in U_1,\\
=0<\xi(x), \ &\text{if} \ x\notin U_1.
\end{cases}
\]

$\bullet$ $k\in B(f,\varepsilon)$, since by ($\ast\ast$) we have that for each $x\in X$,
\[|f(x)-k(x)|=(1-\frac{r}{\delta})(h(x)-f(x))\le(1-\frac{r}{\delta})\varepsilon(x)<\varepsilon(x). \]

\vskip 6pt

It follows from the Claim, and (i) that $h\in B(g,\phi)$, so $|g(x_0)-t|=|g(x_0)-h(x_0)|<\phi(x_0)$, thus, $t \in  (g(x_0)-\phi(x_0),g(x_0)+\phi(x_0))$.

(c) If $t<f(x_0)$, an argument analogous to (b) works, one just needs to define $h_{n}=h_{n-1}-g_{n}$, $k=h+k_0$ in case (1), and  $h_{n+1}=h_{n}-g_{n+1}$, $k=h+\frac{r}{\delta}(h-f)$ in case (2).
\end{proof}

\begin{thm}
Let $X$ be a Tychonoff space such that
\begin{enumerate}
\item either the points of $X$ are $G_{\delta}$,
\item or $X$ is locally countably compact,
\item or $X$ is a $cb$-space.
\end{enumerate}
Then $C_{\Gamma}(X)$ is pseudocomplete.
\end{thm}

\begin{proof} (1-2) For each $n<\omega$, define a base for $C_{\Gamma}(X)$ as
\[\mathcal B_n=\{B(f,\varepsilon)\in\mathcal B_{\Gamma}: f\in C(X), \ \varepsilon(x)\le \frac{1}{2^n} \ \text{for all $x\in X$}\},\]
and choose a sequence $B(f_n,\varepsilon_n)\in\mathcal B_n$ with
$cl_{\Gamma}(B(f_{n+1},\varepsilon_{n+1}))\subseteq B(f_n,\varepsilon_n)$.
 Then for each $x\in X$, and $m>n$,
 \[|f_m(x)-f_{n+1}(x)|<\varepsilon_{n+1}(x)\le \frac{1}{2^{n+1}},\]
 so the sequence $(f_n)$ (uniformly) converges to some $f\in C(X)$.  Then for each $x\in X$, and $n\ge 1$, $|f(x)-f_{n+1}(x)|\le
 \varepsilon_{n+1}(x)$, so by Lemma \ref{key}, $|f(x)-f_n(x)|<\varepsilon_{n}(x)$; thus, $f\in\bigcap_n B(f_n,\varepsilon_n)$.

 (3) By Proposition \ref{w=g}, $C_{\Gamma}(X)=C_w(X)$, and, by \cite[Theorem 3.2]{MC1}, $C_w(X)$ is always pseudocomplete.
\end{proof}

\section{Properties related to 1st and 2nd countability of the graph topology}

It is known that countable compactness of $X$ is also equivalent to various non-completeness properties of $C_{\Gamma}(X)$, such as 1st countability, or metrizability \cite{HH}. We will show that other properties can be added to this list. Recall, that $Z$ is {\it Fr\`echet} iff for each $A\subseteq Z$, and $z\in \overline{A}$ there is a sequence $a_n\in A$ converging to $z$; $Z$ is {\it sequential} iff $A\subseteq Z$ is closed provided $A$ contains the limits of all sequences from $A$;
$Z$ is a {\it $k$-space} iff $A\subseteq Z$ is closed provided $A\cap K$ is closed in $K$ for each compact $K\subseteq Z$;
$Z$ is {\it countably tight} iff for each $A\subseteq Z$, and $z\in \overline{A}$ there is a countable  $B\subseteq A$ with $z\in \overline{B}$; a Tychonoff space $Z$ is a {\it $p$-space} iff there is a sequence $(\mathcal U_n)$ of families of open covers of $Z$ in a compactification of $Z$ such that  $\bigcap_n \bigcup \{U\in\mathcal U_n: z\in U\}\subseteq Z$ whenever $z\in Z$; $Z$ is a {\it $q$-space} iff each $z\in Z$ has an open neighborhood sequence $(U_n)$ such that whenever $z_n\in U_n$, then $(z_n)$ clusters.  For more on these spaces see \cite{En}, \cite{Gr}.

\begin{thm}\label{metr}
 Let $X$ be a Tychonoff space. The following are equivalent:
 \begin{enumerate}
  \item $C_{\Gamma}(X)$ is metrizable,
 \item $C_{\Gamma}(X)$ is a $p$-space,
 \item  $C_{\Gamma}(X)$ is a $q$-space,
  \item $C_{\Gamma}(X)$ is 1st countable,
   \item $C_{\Gamma}(X)$ is a  Fr\`echet space,
   \item  $C_{\Gamma}(X)$ is sequential,
 \item $C_{\Gamma}(X)$ is a $k$-space,
\item  $C_{\Gamma}(X)$ is countably tight,
 \item $X$ is countably compact.
 \end{enumerate}
 \end{thm}

\begin{proof} The implications (9)$\Rightarrow$(1)$\Rightarrow$(2)$\Rightarrow$(3) are well-known;
(3)$\Rightarrow$(4) follows from the observation that a regular $q$-space where the points are $G_{\delta}$ is 1st countable (cf. \cite[Lemma 3.2]{Gr}), which applies, since  $C_{\Gamma}(X)$ is a completely regular, submetrizable space (it contains  the metrizable uniform topology), so the points of $C_{\Gamma}(X)$ are $G_{\delta}$.

The implications (4)$\Rightarrow$(5)$\Rightarrow$(6)$\Rightarrow$(7)  are well-known;
further, for (7)$\Rightarrow$(8), observe that any $k$-space $Z$ with $G_{\delta}$ singletons is countably tight, since otherwise there is some $A\subseteq Z$ so that $B=\bigcup\{ \overline{C}: C\subseteq A \ \text{countable}\}$ is not closed, so there exists some compact $K\subseteq Z$ such that $K\cap B$ is not closed in $K$. Since a  compact space having $G_{\delta}$ singletons is 1st countable,  there exist a sequence $z_n\in K\cap B$ converging to some $z\in K\setminus B$; also, there is a countable $C_n\subseteq A$ with $z_n\in\overline{C_n}$ for each $n$, so $z\in \overline{\bigcup_n C_n}$, which implies $z\in B$, a contradiction.

Finally, to prove (8)$\Rightarrow$(9), assume $X$ is not countably compact. Let   $D=\{x_n: n\ge 1\}$ be a countable closed discrete subset of $X$, and $\{U_n: n\ge 1\}$  a pairwise disjoint sequence of  open neighborhoods of the $x_n$'s. Let $f_0$ be the identically zero function, and define $L=\{g\in C(X): D\subseteq g^{-1}((0,\infty))\}$.

\begin{claimnn}
$f_0\in cl_{\Gamma}(L)$:
\end{claimnn}

Indeed, consider some $B(f_0,\varepsilon)\in \mathcal B_{\Gamma}$. For each $n$, find an open neighborhood $V_n$ of $x_n$ so that
$\overline{V_n} \subseteq U_n\cap \varepsilon^{-1}((\frac23\varepsilon(x_n),\infty))$, and define the continuous function $h_n:\overline{V_n}\to [0,\frac{\varepsilon(x_n)}{n}]$ such that $h_n(x_n)=\frac{\varepsilon(x_n)}{n}$, and $h_n(\overline{V_n}\setminus V_n)=\{0\}$. It is easy to see that  $h\in L\cap B(f_0,\varepsilon)$, where $h:X\to \mathbb R$ is defined via
\[h=\begin{cases} h_n, \ &\text{on} \ \overline{V_n}, \ \text{whenever} \ n\ge 1,\\
0, \ & \text{on} \ X\setminus \bigcup_{n\ge 1}\overline{V_n}.
\end{cases}\]

Since $C_{\Gamma}(X)$ is countably tight, there is a  countable subset $L'=\{f_n: n\ge 1\}$ of $L$ so that $f_0\in cl_{\Gamma}(L')$. Define the function $\eta:X\to (0,\infty)$ via
\[\eta(x)=\begin{cases} \frac{f_n(x)}{k_n}, \ & \text{if} \ x=x_n, \ \text{whenever} \ n\ge 1, k_n \in Z^+, f_n(x_n) < k_n/2\\
\frac12, \ & \text{if} \ x\notin D.
\end{cases}
\]
It follows, that $\eta\in LSC^+(X)$, and $L'\cap B(f_0,\eta)=\emptyset$, a contradiction.
\end{proof}

Now we turn to characterizing 2nd countability and related properties of the graph topology:

\begin{thm} \label{2nd}
Let $X$ be a Tychonoff space. The following are equivalent:
 \begin{enumerate}
  \item $C_{\Gamma}(X)$ is 2nd countable,
 \item $C_{\Gamma}(X)$ has a countable network,
 \item  $C_{\Gamma}(X)$ is separable,
  \item $C_{\Gamma}(X)$ has ccc,
 \item $X$ is  compact and metrizable.
\end{enumerate}
 \end{thm}

\begin{proof}
The implications (1)$\Rightarrow$(2)$\Rightarrow$(3)$\Rightarrow$(4) are trivial; as for (4)$\Rightarrow$(5), note that $\tau_u\subseteq \tau_{\Gamma}$ implies $(C(X),\tau_u)$ has ccc, so it is (metrizable) separable,   which is equivalent to (5) by \cite[Theorem 4.2.4]{MN}. Finally, (5)$\Rightarrow$(1), since $\tau_{\Gamma}=\tau_u$  if $X$ is compact, and then $(C(X),\tau_u)$ is separable and metrizable.
\end{proof}



\end{document}